\documentclass{amsart}
\usepackage{mathtools}
\usepackage{amsfonts,amssymb,amsthm,amsmath}
\usepackage{mathrsfs}
\usepackage{enumerate}
\usepackage{pgf}
\usepackage{url}
\usepackage[all]{xy}
\SelectTips{cm}{}
\usepackage{graphicx}
\usepackage{hyperref}
\usepackage{lineno,hyperref}
\usepackage{cleveref}
\usepackage{float}
\usepackage{comment}
\usepackage{tikz}
\usetikzlibrary{shapes.geometric}
\allowdisplaybreaks

\newcommand{\Z}{\mathbb{Z}}

\newcommand{\N}{\mathbb{N}}

\newcommand{\dist}{\mathrm{dist}}
\newcommand{\diam}{\mathrm{diam}}

\newtheorem{theorem}{Theorem}
\crefname{theorem}{Theorem}{Theorems}

\crefname{proposition}{Proposition}{Propositions}
\newtheorem{lemma}[theorem]{Lemma}
\crefname{lemma}{Lemma}{Lemmas}

\crefname{corollary}{Corollary}{Corollaries}
\newtheorem*{conjecture}{Conjecture}
\crefname{conjecture}{Conjecture}{Conjectures}
\theoremstyle{definition}
\newtheorem*{remark}{Remark}
\crefname{remark}{Remark}{Remarks}
\newtheorem{definition}{Definition}
\crefname{definition}{Definition}{Definitions}
\newtheorem{example}{Example}
\crefname{example}{Example}{Examples}
\crefname{table}{Table}{Tables}
\crefname{section}{Section}{Sections}
\crefname{equation}{Equation}{Equations}
\crefname{figure}{Figure}{Figures}
\crefname{appendix}{Appendix}{Appendices}

\title[Small Gaps Between Three Almost Primes]{Small Gaps Between Three Almost Primes and Almost Prime Powers}

\author[Daniel Goldston, Apoorva Panidapu, and Jordan Schettler]{Daniel A. Goldston, Apoorva Panidapu, and Jordan Schettler}
\date{}

\begin{document}
\maketitle
\begin{abstract}
    A positive integer is called an $E_j$-number if it is the product of $j$ distinct primes. We prove that there are infinitely many triples of $E_2$-numbers within a gap size of $32$ and infinitely many triples of $E_3$-numbers within a gap size of $15$. Assuming the Elliot-Halberstam conjecture for primes and $E_2$-numbers, we can improve these gaps to $12$ and $5$, respectively. We can obtain even smaller gaps for almost primes, almost prime powers, or integers having the same exponent pattern in the their prime factorizations. In particular, if $d(x)$ denotes the number of divisors of $x$, we prove that there are integers $a,b$ with $1\leq a < b \leq 9$ such that $d(x)=d(x+a)=d(x+b) = 192$ for infinitely many $x$. Assuming Elliot-Halberstam, we prove that there are integers $a,b$ with $1\leq a < b \leq 4$ such that $d(x)=d(x+a)=d(x+b) = 24$ for infinitely many $x$.
\end{abstract}
\section{Introduction}

For our purposes, an \emph{almost prime} or \emph{almost prime power} will refer to a positive 
integer with some fixed small number of prime factors 
counted with or without multiplicity, respectively. 
Small gaps between primes and almost primes
became a popular subject of 
research following the
results of the GPY sieve \cite{Gold0} and Yitang Zhang's subsequent proof of bounded gaps between primes \cite{Zhan}.
For a positive integer $x$, let $\Omega(x)$ denote the number of prime 
factors of $x$ counted with multiplicity, and let $\omega(x)$ 
denote the number of prime factors of $x$ counted without 
multiplicity, i.e., $\omega(x)$ is the number of distinct 
primes dividing $x$. We say that $x$ is a 
\emph{$j$-almost prime} when $\Omega(x)=j$, and we say that
$x$ is a \emph{$j$-almost prime power} when $\omega(x)=j$.
We call $x$ an \emph{$E_j$-number} when 
$\Omega(x)=\omega(x)=j$, i.e., $x$ is the product of $j$ distinct primes. Let 
$S_n^{(j)}$, $s_n^{(j)}$, 
$q_n^{(j)}$ denote the $n$th $j$-almost
prime, $n$th $j$-almost prime power, and $n$th 
$E_j$-number, respectively, where the sequences are 
ordered by inequality.
For example, $p_n \mathrel{\mathop:}= q^{(1)}_n = S^{(1)}_n$ is the $n$th 
prime number, $s^{(1)}_n$ is the $n$th prime power, and

\[S_1^{(2)} = 4, ~ S_2^{(2)} = 6, ~ S_3^{(2)} = 9, ~ S_4^{(2)} = 10, ~\ldots, \]

\[s_1^{(2)} = 6, ~ s_2^{(2)} = 10, ~ s_3^{(2)} = 12, ~  s_4^{(2)} = 14, ~ \ldots, \]

\[q_1^{(2)} = 6, ~ q_2^{(2)} = 10, ~ q_3^{(2)} = 14, ~ q_4^{(2)} = 15, ~ \ldots . \]

It is known (see \cite{Mayn} for $j=1$ and \cite{Thor} for $j>1$) that for any positive integers $j$ and $\nu$, we have
$$\liminf_{n\rightarrow \infty} (q_{n+\nu}^{(j)} - q_n^{(j)}) < \infty.$$
Thus there are bounded gaps containing infinitely often 
$\nu+1$ members of 
any of the sequences $S_n^{(j)}$, $s_n^{(j)}$, 
$q_n^{(j)}$
since any $E_j$-number is both a $j$-almost prime 
and a $j$-almost prime power. When $\nu = 1$, we 
have the following results unconditionally:
$$\liminf_{n\rightarrow \infty} (p_{n+1} - p_n) \leq 246, \mbox{\hspace{0.1 in} \cite{Poly}}$$
$$\liminf_{n\rightarrow \infty} (q_{n+1}^{(2)} - q_n^{(2)}) \leq 6, \mbox{\hspace{0.1 in} \cite{Gold1}}$$
$$\liminf_{n\rightarrow \infty} (S_{n+1}^{(3)} - S_n^{(3)})\leq \liminf_{n\rightarrow \infty} (q_{n+1}^{(3)} - q_n^{(3)}) \leq 2, \mbox{\hspace{0.1 in} \cite{Gold2}}$$
$$\liminf_{n\rightarrow \infty} (s_{n+1}^{(3)} - s_n^{(3)}) = 1, \mbox{\hspace{0.1 in} \cite{Gold2}}$$
and for $j\geq 4$, we get
$$\liminf_{n\rightarrow \infty} (s_{n+1}^{(j)} - s_n^{(j)})=\liminf_{n\rightarrow \infty} =(S_{n+1}^{(j)} - S_n^{(j)}) = 1. \mbox{\hspace{0.1 in} \cite{Gold2}}$$
We can can also get sharper bounds by assuming some widely 
believed conjectures. For example, if there are infinitely 
many Mersenne primes $2^p-1$, then 
$$\liminf_{n\rightarrow \infty} (s_{n+1}^{(1)} - s_n^{(1)}) = 1.$$
If we assume the Elliott-Halberstam conjecture for primes (see \cref{ggpy}), then we get a much smaller gap for primes:
$$\liminf_{n\rightarrow \infty} (p_{n+1} - p_n) \leq 12. \mbox{\hspace{0.1 in} \cite{Mayn}}$$
In fact, a generalized Elliott-Halberstam conjecture can 
reduce this gap to $6$, while the still unproven twin prime 
conjecture asserts that $\liminf_{n\rightarrow \infty} (p_{n+1} - p_n) = 2$. The twin prime conjecture follows from a special case of the prime $k$-tuple conjecture, the first Hardy-Littlewood conjecture, which we will describe in the next section.

In this paper, we are interested in the case $\nu=2$ and $j\geq 2$. This 
means we want to find small gaps which contain three $j$-almost primes or $j$-almost prime powers infinitely often. For the primes themselves (i.e., when $j=1$), the $k$-tuple conjecture for $k=\nu+1=3$ implies that
$$\liminf_{n\rightarrow \infty} (p_{n+2} - p_n) = 6.$$
Unconditionally, we know
$$\liminf_{n\rightarrow \infty} (p_{n+2} - p_n) \leq 398130, \mbox{\hspace{0.1 in} \cite{Poly}}$$
and assuming the Elliott-Halberstam conjecture for primes we can improve this to
$$\liminf_{n\rightarrow \infty} (p_{n+2} - p_n) \leq 270. \mbox{\hspace{0.1 in} \cite{Poly}}$$
For three $j$-almost primes or $j$-almost prime powers with $j\geq 2$, we can get much smaller gaps. For instance, assuming the Elliott-Halberstam conjecture for primes and $E_2$-numbers, Sono proved
\begin{equation}\label{Sonoresult}
\liminf_{n\rightarrow \infty} (q_{n+2}^{(2)}-q_n^{(2)}) \leq 12 \mbox{\hspace{0.1 in} \cite{Sono}}
\end{equation}
and
$$\liminf_{n\rightarrow \infty} (r_{n+2}^{(2)}-r_n^{(2)}) \leq 6 \mbox{\hspace{0.1 in} \cite{Sono}}$$
where $r_n^{(2)}$ denotes the $n$th $\mathcal{P}_2$-number (i.e., a prime or $E_2$). \\

Sono used a multi-dimensional, Maynard-Tao sieve to obtain his results. Similar techniques have been applied to $E_j$-numbers for $j > 2$ as in \cite{Liu}. Here we use the GGPY sieve to give an alternate proof of Sono's inequality in \cref{Sonoresult} and also prove an unconditional version which has not been previously derived. Let $\mathrm{EH}(\mathcal{P}, \mathcal{E}_2)$ denote the Elliot-Halberstam conjecture for primes and $E_2$-numbers.

\begin{theorem}\label{E2}
We have
$$\liminf_{n\rightarrow \infty} (q_{n+2}^{(2)}-q_n^{(2)}) \leq 32.$$
If we assume $\mathrm{EH}(\mathcal{P}, \mathcal{E}_2)$, then
$$\liminf_{n\rightarrow \infty} (q_{n+2}^{(2)}-q_n^{(2)}) \leq 12.$$
\end{theorem}

We use combinatorial methods which complement the sieve results to prove our main theorems.

\begin{theorem}\label{E3}
We have
\[\liminf_{n\rightarrow \infty} (q^{(3)}_{n+2} - q^{(3)}_n)\leq 15. \]
If we assume $\mathrm{EH}(\mathcal{P}, \mathcal{E}_2)$, then
\[\liminf_{n\rightarrow \infty} (q^{(3)}_{n+2} - q^{(3)}_n)\leq 5. \]
\end{theorem}

\begin{theorem}\label{S3}
We have
\[\liminf_{n\rightarrow \infty} (s_{n+2}^{(4)} - s_n^{(4)}), \, \liminf_{n\rightarrow \infty} (S_{n+2}^{(5)} - S_n^{(5)})\leq 9.\]
If we assume $\mathrm{EH}(\mathcal{P}, \mathcal{E}_2)$, then
\[\liminf_{n\rightarrow \infty} (s_{n+2}^{(3)} - s_n^{(3)}),\, \liminf_{n\rightarrow \infty} (S_{n+2}^{(4)} - S_n^{(4)})\leq 4. \]
\end{theorem}

We will also show that \cref{S3} is optimal in the sense that smaller gap sizes cannot be obtained with current methods. See \cref{diametertheorem}.

Let $d(x)$ denote the number of divisors of $x$. Then the twin prime conjecture is equivalent to the statement $d(x)=d(x+2)=2$ for infinitely many $x$. The authors and their collaborators showed in \cite{GGPPSY} that for every integer $n>0$, there are infinitely many $x$ such that
$x$ and $x+n$ have the same fixed \emph{exponent pattern}, i.e., multiset of exponents in the prime factorization; in particular, $d(x)=d(x+n)=c$ for some fixed $c$ (depending on $n$) and infinitely many $x$. For example, both $180 = 2^2\cdot 3^2 \cdot 5^1$ and $300 = 2^2\cdot 3^1 \cdot 5^2$ have exponent pattern $\{2,2,1\}$ with $d(180)=d(300)=18$. Little is known about multiple shifts $x<x+a<x+b$. It is conjectured that $d(x)=d(x+1)=d(x+2)=4$ for infinitely many $x$, but no statement of the form $d(x)= d(x+a)=d(x+b)$ for infinitely many $x$ has been established for any particular $a$, $b$. Here we prove the following.

\begin{theorem}\label{Div}
There are integers $a, b$ with $1\leq a<b\leq 9$ such that $x$, $x+a$, and $x+b$ all have exponent pattern $\{3,2,1,1,1,1\}$ for infinitely many $x$, so here
$d(x)=d(x+a)=d(x+b)=192$.

If we assume $\mathrm{EH}(\mathcal{P}, \mathcal{E}_2)$, there are integers $a, b$ with $1\leq a < b \leq 4$ such that $x$, $x+a$, and $x+b$ all have exponent pattern $\{2,1,1,1\}$ for infinitely many $x$, so here $d(x)=d(x+a)=d(x+b)=24$.
\end{theorem}

\section{Notation and Preliminaries}

\begin{definition}
We define a \emph{linear form} to be an expression $L(m) = am+b$ where $a$ and $b$ are integers with $a>0$, and we say
$L$ is \emph{reduced} when $(a,b)=1$. We will view a linear form as both a polynomial and a function of $m\in\N$.
\end{definition}

Dirichlet's theorem on primes in arithmetic progressions guarantees that any reduced linear form will assume prime values infinitely often. In fact, if $L(m) = am+b$ is reduced and we define
\[\pi(x; a, b) = \#\{p\leq x: p = L(m) \mbox{ is prime for some } m\in\N\}\]
then we have the asymptotic
\begin{equation}\label{Dir}
\pi(x; a, b) \sim
    \frac{\mathrm{li}(x)}{\varphi(a)} \end{equation}
as $x\rightarrow \infty$ where $\mathrm{li}(x)$ is the logarithmic integral and $\varphi(a)$ denotes the Euler 
phi-function. Thus if we define the related 
prime counting function
\[\pi(x; L(m)) = \#\{ m\in \N : m\leq x \mbox{ and } L(m) \mbox{ is prime}\}\]
we get
\begin{equation}\label{Dirichlet}
\pi(x; L(m)) \sim
    \frac{a}{\varphi(a)}\cdot \frac{x}{\log x} \end{equation}
as $x\rightarrow \infty$.
The question of when two or 
more reduced linear forms simultaneously assume prime values infinitely often is addressed by the unproven Hardy-Littlewood prime $k$-tuples conjecture,
which states roughly that linear forms will simultaneously assume prime values infinitely often unless there is some obvious congruence preventing it. For
example, the triple of forms $m$, $m+2$, and $m+10$ cannot be simultaneously prime infinitely often since their product is congruent to zero modulo 3
for any $m$. To state a more precise form of the conjecture
we need a few definitions. Let $\mathcal{L}$ denote a $k$-tuple of linear forms, i.e., a sequence of $k$ distinct forms $L_1=a_1m+b_1$, $L_2 = a_2m+b_2$, $\ldots$, $L_k=a_km+b_k$.
Next, define a singular series for $\mathcal{L}$ via
\begin{equation*} \label{singular}
\mathfrak{S} (\mathcal{L})
 =
 \prod_{p~\mbox{\tiny prime}} \left(1 - \frac{1}{p}\right)^{-k} \left(1 - \frac{\nu_{\mathcal{L}}(p)}{p}\right),
\end{equation*}
where
$\nu_{\mathcal{L}}(p)$ denotes the number of solutions $m\in \{1, 2, \ldots, p\}$ to $\prod_{i=1}^k L_i(m) \equiv 0 \pmod{p}$. Note that if $\nu_{\mathcal{L}}(p) = p$ for some prime number $p$, then $\mathfrak{S}(\mathcal{L}) = 0$. When $\nu_{\mathcal{L}}(p) < p$, we say $\mathcal{L}$ is \emph{$p$-admissible}. If $\mathcal{L}$ is $p$-admissible for all prime numbers $p$, then  $\mathfrak{S}(\mathcal{L}) \neq 0$, and in this case the $k$-tuple $\mathcal{L}$ is called \emph{admissible}. Note that admissibility of $\mathcal{L}$ is equivalent to having all forms in $\mathcal{L}$ reduced and having $p$-admissibility for all primes $p\leq k$. We define a prime counting function for the $k$-tuple $\mathcal{L}$ as
\[\pi(x; \mathcal{L}) = \#\{ m\in \N : m\leq x \mbox{ and } L_i(m) \mbox{ is prime for all } i\}. \]
We can now state a generalization of the asymptotic in \cref{Dirichlet}.

\begin{conjecture}[Hardy-Littlewood Prime $k$-Tuple Conjecture]\label{HL}
For each fixed $k\geq 1$ and admissible $k$-tuple $\mathcal{L}$ as above, we have
\begin{equation} \label{hlconj}
\pi (x; \mathcal{L})
 \sim \mathfrak{S}(\mathcal{L}) \cdot \frac{x}{(\log x)^k}
\end{equation}
as $x\rightarrow \infty$.
\end{conjecture}

While a proof of the prime $k$-tuple conjecture appears beyond 
our current state of knowledge (even in the cases $k=2$ or $3$), we do know
the following major result which was proven in \cite{Gold1}.

\begin{theorem}\label{GGPY09}
Let $C$ be any constant and $\nu$ be any positive integer. Then for $k$ sufficiently large, every admissible $k$-tuple of linear forms has $\nu+1$ among them which infinitely often take $E_2$-numbers simultaneously as values with both prime factors above $C$.
\end{theorem}
% The Hardy-Littlewood prime 3-tuple conjecture implies that the linear forms $m, m+2, m+6$ are all simultaneously prime for infinitely many $m$, so here $d(x) = d(x+2) = d(x+6) = 2$ for infinitely many positive integers $x$. What about non-admissible triples? They can never be prime tuples infinitely often. For example, the triple $m,m+1,m+2$ can never have more than two primes in it. In this situation,
% we can use relations among linear forms to show that the
% divisor function should assume a common value infinitely often for such triples.

The authors in \cite{Gold1} showed further that we may take $k=3$ when $\nu=1$. This means every admissible triple contains two forms
which simultaneously assume $E_2$-numbers with large prime factors as values infinitely often. For $\nu=2$ we have the following.

\begin{theorem}\label{thm:nu=2}
Let $C$ be any constant. Then every admissible $10$-tuple of linear forms has $3$ among them which infinitely often take $E_2$-numbers simultaneously as values with both prime factors above $C$. If we assume $\mathrm{EH}(\mathcal{P}, \mathcal{E}_2)$, then the same conclusion holds for every admissible $5$-tuple.
\end{theorem}

The proof of this result, which is largely computational, is postponed until \cref{ggpy}, the last in the paper. However, \cref{thm:nu=2} allows us to give a quick proof of \cref{E2}.

\begin{proof}[Proof of \cref{E2}]
The $10$-tuple $m$, $m+2$, $m+6$, $m+8$, $m+12$, $m+18$, $m+20$, $m+26$, $m+30$, $m+32$ is admissible.  The $5$-tuple $m$, $m+2$, $m+6$, $m+8$, $m+12$ is admissible.
\end{proof}

We now develop the combinatorial tools needed to prove our main results, \cref{E3,S3,Div}.

\section{Relation Diagrams and Adjoining Transformations}

\cref{thm:nu=2} says that for $k\geq 10$ (or $k
\geq 5$ assuming $\mathrm{EH}(\mathcal{P}, \mathcal{E}_2)$), every admissible $k$-tuple 
$\mathcal{L}=(L_i)_{i=1}^k$ contains $3$ forms $L_h$, $L_i$, $L_j$ such that $L_h(m)$, $L_i(m)$, $L_j(m)$ are all $E_2$-numbers with large prime factors for each $m$ belonging to an infinite set $M$ of positive integers.
We can take integer combinations of these three forms to get small gaps between three almost primes. For example, if
$L_h=10m+1$, $L_i=15m+2$, $L_j=6m+1$,
then $3L_h = 2L_i - 1 = 5L_j - 2$,
so for each $m\in M$ we would get three $E_3$ numbers $3L_h(m)$, $2L_i(m)$, $5L_j(m)$
within a gap size of $2$. The difficulty here is that \cref{thm:nu=2}
does not specify which triple of forms $L_h$, $L_i$, $L_j$ in our $k$-tuple does the job; the theorem only guarantees that such a triple exists.
Thus we must find admissible $k$-tuples where every triple of forms contained in it
gives us small gaps for three of the same kind of almost primes.

\begin{definition}
Given two linear 
forms $L_1$ and $L_2$, a \emph{relation from $L_1$ to $L_2$} is an equation of the form
\begin{equation*}
    c_2\cdot L_2 - c_1 \cdot L_1 = r
\end{equation*} where $c_1$, $c_2$, and $r$ are all positive integers. We denote such a relation by
\begin{center}
\begin{tikzpicture}
  \coordinate [label={[xshift=0cm, yshift=-.3cm]$L_1$}] (L1) at (-1.5cm,0cm);
  \coordinate [label={[xshift=0cm, yshift=-.3cm]$L_2$}] (L2) at (1.5cm,0cm);
    ;
    \draw [thick, |->, shorten >= 9pt, shorten <= 9pt] (L1) -- (L2) node[midway,above] {$r$} node[pos=0.2,above] {\scriptsize $(c_1)$} node[pos=0.8,above] {\scriptsize $(c_2)$};
  \end{tikzpicture}
\end{center}
and we call $c_1$, $c_2$ the \emph{relation coefficients} and $r$ the \emph{relation value}. The \emph{distance} between $L_1$ and $L_2$, denoted by $\dist(L_1, L_2)$, is defined to be the minimum value of $r$ in any such relation.

A \emph{relation diagram} is a directed graph where the vertices are reduced linear forms and the edges are relations. Here we will always assume our diagrams are simple, i.e., have at most one edge between each pair of forms.

A \emph{triangle relation} for a triple of forms is a relation diagram of the form seen in \cref{fig:triangle}.
\begin{figure}[H]
 \begin{tikzpicture}
  \coordinate [label={[xshift=-0.1cm, yshift=-.3cm]$L_1$}] (L2) at (-2cm,-1.cm);
  \coordinate [label={[xshift=0cm, yshift=-.25cm]$L_2$}] (L1) at (0cm,1.0cm);
  \coordinate [label={[xshift=0.1cm, yshift=-.3cm]$L_3$}] (L3) at (2cm,-1.0cm);
    ;
    \draw [thick, |->, shorten >= 9pt, shorten <= 9pt] (L2) -- (L1) node[midway,above left] {$r_1$} node[pos=0.2,above left] {\scriptsize $(c_1)$} node[pos=0.8,above left] {\scriptsize $(c_2)$};
    \draw [thick, |->, shorten >= 9pt, shorten <= 9pt] (L1) -- (L3) node[midway,above right] {$r_2$} node[pos=0.2,above right] {\scriptsize $(c_2)$} node[pos=0.8,above right] {\scriptsize $(c_3)$};
    \draw [thick, |->, shorten >= 10pt, shorten <= 10pt] (L2) -- (L3) node[midway,below] {$r_1+r_2$} node[pos=0.15,below] {\scriptsize $(c_1)$} node[pos=0.85,below] {\scriptsize $(c_3)$};
  \end{tikzpicture}
    \caption{A General Triangle Relation}
    \label{fig:triangle}
\end{figure}
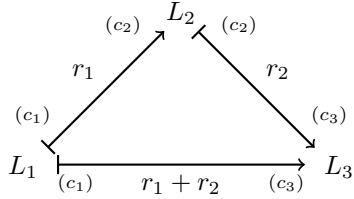
\noindent The 
\emph{diameter} of the triple, denoted by $\diam(L_1,L_2,L_3)$, is defined to be the minimum value of the sum
$r_1+r_2$ in any such triangle relation.
%Note that the diameter of the triple
%is always a multiple of $\dist(L_1,L_3)$ and is bounded below by $\dist(L_1,L_2)+\dist(L_2, L_3) \geq 2$.
\end{definition}
% \begin{definition}
% When there is a relation from one linear form $L_i$ to another $L_j$, we simply write $L_i\longmapsto L_j$. It is straightforward to show that every relation from $L_i$ to $L_j$ is a positive integer multiple of a minimal relation between the forms, which is the unique relation with minimal relation value.
% The minimal relation value is called the \emph{distance} from $L_i$ to $L_j$, denoted $\dist(L_i,L_j)$. If we let $[a,b]$ denote the least common multiple of positive integers $a$, $b$, then we have the following formula for the distance between two forms.
% \end{definition}

\begin{lemma}\label{lem:distance}
There is a strict total order on reduced forms given by
$L_i\longmapsto L_j$ whenever there is a relation from $L_i$ to $L_j$. Given a triple of linear forms $L_i(m) = a_im+b_i$ for $i=1,2,3$ with $L_1 \longmapsto L_2 \longmapsto L_3$ we have
$$\dist(L_1,L_2)=[a_1,a_2]\left(\frac{b_2}{a_2}-\frac{b_1}{a_1}\right)$$
and
$$\diam(L_1,L_2,L_3) = [a_1,a_2,a_3] \left(\frac{b_3}{a_3} - \frac{b_1}{a_1} \right)$$
where $[a_1,a_2]$ and $[a_1,a_2,a_3]$ denote least common multiples.
\end{lemma}
\begin{proof}
Note that since we assume linear coefficients and relation values are positive, $L_i\longmapsto L_j$ is equivalent to the inequality $b_i/a_i < b_j/a_j$. This gives us transitivity. Trichotomy follows from this along with the assumption the forms are reduced. To prove the formula for distance, we need to minimize $r = c_2b_2-c_1b_1 = v(b_2/a_2-b_1/a_1)$ subject to the condition $v=c_2a_2=c_1a_1$ for positive integers $c_1, c_2$. The smallest value for $v$ is $[a_1,a_2]$. To prove the formula for diameter, we need to minimize $r_1+r_2 = c_3b_3-c_1b_1 = v(b_3/a_3-b_1/a_1)$ subject to the condition $v=c_3a_3=c_2a_2=c_1a_1$ for positive integers $c_1, c_2, c_3$. The smallest value for $v$ is $[a_1,a_2,a_3]$.
 \end{proof}

\begin{definition}
A relation diagram on a $k$-tuple $\mathcal{L}= (L_i)_{i=1}^k$ of linear forms is called \emph{consistent} if the graph is complete with edges of the form
\begin{center}
\begin{tikzpicture}
  \coordinate [label={[xshift=0cm, yshift=-.3cm]$L_i$}] (L1) at (-1.5cm,0cm);
  \coordinate [label={[xshift=0cm, yshift=-.35cm]$L_j$}] (L2) at (1.5cm,0cm);
    ;
    \draw [thick, |->, shorten >= 9pt, shorten <= 9pt] (L1) -- (L2) node[midway,above] {$r_{i,j}$} node[pos=0.2,above] {\scriptsize $(c_i)$} node[pos=0.8,above] {\scriptsize $(c_j)$};
  \end{tikzpicture}
\end{center}
so each form $L_i$ has the same relation coefficient $c_i$ in all of its edges. Triangle relations, for instance, are consistent, and the induced subgraph for any triple in a consistent diagram is a triangle relation. We say that a consistent diagram is also \emph{$f$-compatible} for some arithmetic function $f$ when $f(c_i)$ is constant across all $i$.
\end{definition}
\begin{definition}
Let $f$ be an arithmetic function. We say $f$ is a \emph{function on exponent patterns} if $f(n)=f(m)$ whenever $m$ and $n$ have the same exponent pattern. We say $f$ is \emph{homomorphic} if the following property holds: $f(a)=f(b)$ and $f(c)=f(d)$ with $(a,c)=1=(b,d)$ implies $f(ac) = f(bd)$.
\end{definition}
\begin{remark}
Note that any multiplicative or additive arithmetic function
is homomorphic. Examples of arithmetic functions which are both homomorphic and functions on exponent patterns include $d$ (number of divisors), $\omega$ (number of distinct prime factors), $\Omega$ (number of prime factors), or the function $h$ defined by $h(n) =$ least positive integer having the same exponent pattern as $n$. Here the function $h$ is neither additive nor multiplicative. The average exponent of $n$ is an example of a function on exponent patterns which is not homomorphic.
\end{remark}

\begin{lemma}\label{pattern}
Suppose $k\geq 10$ (or $k\geq 5$ assuming $\mathrm{EH}(\mathcal{P}, \mathcal{E}_2)$) and that we have an admissible $k$-tuple $\mathcal{L}= (L_i)_{i=1}^k$ of linear forms satisfying
a consistent relation diagram with edges
\begin{center}
\begin{tikzpicture}
  \coordinate [label={[xshift=0cm, yshift=-.3cm]$L_i$}] (L1) at (-1.5cm,0cm);
  \coordinate [label={[xshift=0cm, yshift=-.35cm]$L_j$}] (L2) at (1.5cm,0cm);
    ;
    \draw [thick, |->, shorten >= 9pt, shorten <= 9pt] (L1) -- (L2) node[midway,above] {$r_{i,j}$} node[pos=0.2,above] {\scriptsize $(c_i)$} node[pos=0.8,above] {\scriptsize $(c_j)$};
  \end{tikzpicture}
\end{center}
whenever $i<j$.
If the diagram is also  $f$-compatible for some homomorphic function $f$ on exponent patterns, then there is a constant $c$ and integers $a$, $b$ with $$r_{\mathrm{min}} \leq a<b\leq r_{\mathrm{max}}$$
such that
\[f(x) = f(x+a) = f(x+b)=c\]
for infinitely many $x$ where $r_{\mathrm{min}}$ (resp. $r_{\mathrm{max}}$) denotes the minimum (resp. maximum) relation value in the diagram.
\end{lemma}

\begin{proof}
By \cref{thm:nu=2}, there are $3$ 
forms $L_h$, $L_i$, $L_j$ in the admissible $k$-tuple $\mathcal{L}$ such that $L_h(m)$, $L_i(m)$, $L_j(m)$
are all $E_2$-numbers for each $m$ belonging to an infinite set $M$ of positive integers. If we take
\begin{equation}\label{x}
x=c_{h}L_h(m)
\end{equation}
for each $m\in M$, then
\begin{equation}\label{xa}
    x+r_{h,i} = c_{i}L_i(m)
\end{equation} and, since the diagram is consistent,
\begin{equation}\label{xb}
    x+r_{h,j} = x+r_{h,i}+r_{i,j} = c_iL_i(m) + r_{i,j} = c_{j}L_j(m).
\end{equation}
Also, $f(c_{h}) = f(c_i) = f(c_{j})$ since the diagram is $f$-compatible and $f(L_h(m))= f(L_i(m))= f(L_j(m))$ is constant
for all $m\in M$
since $f$ is a function on exponent patterns and $L_h(m)$, $L_i(m)$, and $L_j(m)$ are all $E_2$-numbers. In fact, we may assume the prime factors of
$L_h(m)$, $L_i(m)$, and $L_j(m)$ are all larger than any of the relation coefficients in the diagram, so since $f$ is homomorphic, we get $c \mathrel{\mathop:}=f(c_{h}L_h(m)) = f(c_{i}L_i(m)) = f(c_{j}L_j(m))$ is a constant across all $m
\in M$. In fact, $c$ is also independent of $h,i,j$ by $f$-compatibility. The result now follows by inspection of \cref{x,xa,xb}.
\end{proof}

\begin{example}\label{diagram}
Consider the relation diagram in \cref{fig:P3}.
\begin{figure}
 \begin{tikzpicture}
    \node[
      regular polygon,
      regular polygon sides=5,
      minimum width=80mm,
    ] (L) {}
      (L.corner 1) node (L1) {$2m+ 1$}
      (L.corner 2) node (L2) {$3m + 2$}
      (L.corner 3) node (L3) {$6m+ 5$}
      (L.corner 4) node (L4) {$6m+7$}
      (L.corner 5) node (L5) {$3m+ 4$}
    ;
    \draw [thick, |->, shorten >= 6pt, shorten <= 6pt] (L1) -- (L2) node[midway,above left] {1} node[pos=0.1,above left] {\scriptsize (3)} node[pos=0.9,above left] {\scriptsize (2)};
    \draw [thick, |->, shorten >= 6pt, shorten <= 6pt] (L1) -- (L3) node[midway,above left] {2} node[pos=0.13,above left] {\scriptsize (3)} node[pos=0.9,above left] {\scriptsize (1)};
    \draw [thick, |->, shorten >= 6pt, shorten <= 6pt] (L1) -- (L4) node[midway,above right] {4} node[pos=0.13,above right] {\scriptsize (3)} node[pos=0.9,above right] {\scriptsize (1)};
    \draw [thick, |->, shorten >= 6pt, shorten <= 6pt] (L1) -- (L5) node[midway,above right] {5} node[pos=0.1,above right] {\scriptsize (3)} node[pos=0.9,above right] {\scriptsize (2)};
    \draw [thick, |->, shorten >= 6pt, shorten <= 6pt] (L2) -- (L3) node[midway,below left] {1} node[pos=0.05,below left] {\scriptsize (2)} node[pos=0.85,below left] {\scriptsize (1)};
    \draw [thick, |->, shorten >= 6pt, shorten <= 6pt] (L2) -- (L4) node[midway,below left] {3} node[pos=0.11,below left] {\scriptsize (2)} node[pos=0.9,below left] {\scriptsize (1)};
    \draw [thick, |->, shorten >= 6pt, shorten <= 6pt] (L2) -- (L5) node[midway,above] {4} node[pos=0.1,above] {\scriptsize (2)} node[pos=0.9,above] {\scriptsize (2)};
    \draw [thick, |->, shorten >= 4pt, shorten <= 4pt] (L3) -- (L4) node[midway,below] {2} node[pos=0.11,below] {\scriptsize (1)} node[pos=0.9,below] {\scriptsize (1)};
    \draw [thick, |->, shorten >= 6pt, shorten <= 6pt] (L3) -- (L5) node[midway,below right] {3} node[pos=0.11,below right] {\scriptsize (1)} node[pos=0.9,below right] {\scriptsize (2)};
    \draw [thick, |->, shorten >= 6pt, shorten <= 6pt] (L4) -- (L5) node[midway,below right] {1} node[pos=0.13,below right] {\scriptsize (1)} node[pos=0.95,below right] {\scriptsize (2)};
  \end{tikzpicture}
    \caption{A Consistent Diagram on an Admissible $5$-Tuple}
    \label{fig:P3}
\end{figure}
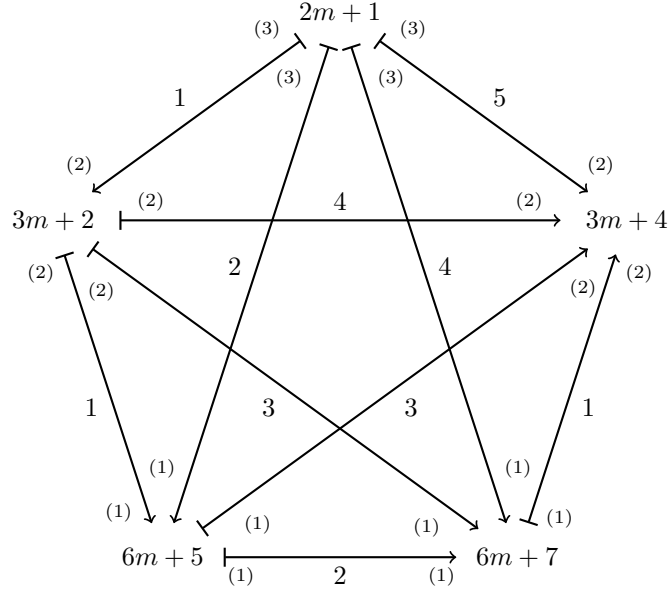
The $5$-tuple is admissible and the diagram is consistent, but the diagram is not $f$-compatible for any $f\in \{d,\omega,\Omega,h\}$. However, all the relation coefficients are either $1$ or a prime, and we show later how to ``adjoin'' primes to get a consistent diagram on an admissible tuple with all prime relation coefficients, which will yield a new diagram that is both $\Omega$- and $\omega$-compatible.
\end{example}

\begin{example}
Consider the $6$-tuple $24m +5$, $90m+ 19$, $288m+ 61$, $33m+ 7$, $80m+ 17$, $108m+ 23$. The tuple is admissible and the distance between any two forms here is $1$, yet $\max_{h<i<j} \diam(L_h,L_i,L_j)=20$. In general,
the maximum diameter for triples within a $k$-tuple can never be 2 (the minimal value for diameters of triples) unless $k=3$. A general lower bound on the maximum diameter follows from Graham's Conjecture \cite{Grah}, which is now a theorem proven by Balasubramanian and Soundararajan in \cite{Bala}.
\end{example}

\begin{theorem}[Graham's Conjecture]\label{graham}
Suppose $m_1$, $m_2$, $\ldots$, $m_n$ are distinct positive integers. Then there are $i$, $j$ such that
\begin{equation*}
    \frac{m_j}{(m_i,m_j)} \geq n.
\end{equation*}
\end{theorem}

\begin{theorem}\label{diametertheorem}
Given an integer $k \geq 3$ and a $k$-tuple of linear forms $L_1 \longmapsto L_2 \longmapsto \cdots \longmapsto L_k$,
we have
$$\max_{h<i<j} \diam(L_h,L_i,L_j) \geq k-1.$$
\end{theorem}

\begin{proof}
Write $L_i(m) = a_im+b_i$ so that $b_1/a_1 < b_2/a_2 < \ldots < b_k/a_k$. Take $A = [a_1,a_2, \ldots, a_k]$ and for $i\geq 2$ define
\begin{equation*}
    m_i = A\left(\frac{b_i}{a_i} - \frac{b_1}{a_1} \right).
\end{equation*}
Each $m_i$ is an integer and $0<m_2<m_3<\ldots <m_k$. Thus by \cref{graham} there are $i$, $j$ such that
\begin{equation*}
    \frac{m_j}{(m_i,m_j)} \geq k-1.
\end{equation*}
Note that $A/[a_1,a_i,a_j]$ is a common divisor of $m_i$, $m_j$, so by \cref{lem:distance}
\begin{align*}
    \diam(L_1,L_i,L_j) &\geq  [a_1,a_i,a_j]\left(\frac{b_j}{a_j} - \frac{b_1}{a_1} \right) =\frac{m_j}{A/[a_1,a_i,a_j]} \\
    &\geq \frac{m_j}{(m_i,m_j)} \geq k-1.
\end{align*}
\end{proof}

\begin{remark}
\cref{diametertheorem} illustrates that there are fundamental restrictions for the gap sizes that we can obtain with these methods. In particular, the upper bound $r_{\mathrm{max}}$ for the gap size $b$ in the multiple shift $x$, $x+a$, $x+b$ of  \cref{pattern} satisfies
$$r_{\mathrm{max}} \geq \max_{h<i<j} \diam(L_h,L_i,L_j) \geq k-1.$$
On the other hand, the gap sizes we get are much better than what one obtains from using sieve methods alone. For example, \cref{E2} says there are infinitely many triples of $E_2$-numbers within a gap size of $32$ (or $12$ assuming $\mathrm{EH}(\mathcal{P}, \mathcal{E}_2)$), yet \cref{E3} says there are infinitely many triples of $E_3$-numbers within a gap size of $15$ (or $5$ assuming $\mathrm{EH}(\mathcal{P}, \mathcal{E}_2)$). The reason this works is that we are using unsifted prime factors in our $E_3$-numbers,
which means we allow small prime factors like $2,3,5,7$ in order to get small maximum diameters. The main tool we use is a certain transformation on tuples which preserves relation values in diagrams, can make a non-admissible tuple admissible, and produces an $f$-compatible diagram while maintaining admissibility.
\end{remark}

\begin{definition}
For $A,B\in \Z$ with $A>0$, we define the \emph{adjoining transformation} $T_{A,B}$ from the set of reduced linear forms to itself via $T_{A,B}(L(m)) = L(Am+B)/g_L$ where $g_L=(aA, aB+b)$ is called the \emph{adjoining factor} for the form $L(m)=am+b$ under the transformation $T_{A,B}$.
\end{definition}

\begin{lemma}\label{adjointhm}
Suppose $\mathcal{L}= (L_i)_{i=1}^k$ is a $k$-tuple of reduced linear forms $L_i(m)=a_im+b_i$ satisfying
a relation diagram with edges
\begin{center}
\begin{tikzpicture}
  \coordinate [label={[xshift=0cm, yshift=-.3cm]$L_i$}] (L1) at (-1.5cm,0cm);
  \coordinate [label={[xshift=0cm, yshift=-.35cm]$L_j$}] (L2) at (1.5cm,0cm);
    ;
    \draw [thick, |->, shorten >= 9pt, shorten <= 9pt] (L1) -- (L2) node[midway,above] {$r_{i,j}$} node[pos=0.2,above] {\scriptsize $(c_{i,j})$} node[pos=0.8,above] {\scriptsize $(c_{j,i})$};
  \end{tikzpicture}
\end{center}
Then the following hold:
\begin{enumerate}[(a)]
    \item For every adjoining transformation $T_{A,B}$, the $k$-tuple $T_{A,B}(\mathcal{L}) = (T_{A,B}(L_i))_{i=1}^k$ satisfies a relation diagram with edges
    \begin{center}
\begin{tikzpicture}
  \coordinate [label={[xshift=-0.6cm, yshift=-.35cm]$T_{A,B}(L_i)$}] (L1) at (-1.75cm,0cm);
  \coordinate [label={[xshift=0.6cm, yshift=-.35cm]$T_{A,B}(L_j)$}] (L2) at (1.75cm,0cm);
    ;
    \draw [thick, |->, shorten >= 9pt, shorten <= 9pt] (L1) -- (L2) node[midway,above] {$r_{i,j}$} node[pos=0.2,above] {\scriptsize $(g_ic_{i,j})$} node[pos=0.8,above] {\scriptsize $(g_jc_{j,i})$};
  \end{tikzpicture}
\end{center}
where the $g_i$ are the adjoining factors for $L_i$ under $T_{A,B}$. Thus if the diagram for $\mathcal{L}$ is consistent, so is the diagram for $T_{A,B}(\mathcal{L})$.
\item We have inequalities
\begin{eqnarray*}
    &\dist(T_{A,B}(L_i),T_{A,B}(L_j))\leq \dist(L_i,L_j) \\
   &\diam(T_{A,B}(L_h),T_{A,B}(L_i),T_{A,B}(L_j))\leq \diam(L_h,L_i,L_j)
\end{eqnarray*}
\item If $\mathcal{L}$ is admissible, then for every choice of positive integers $g_1$, $g_2$, $\ldots$, $g_k$ such that $(g_i,a_i) = (g_i, a_ib_j-a_jb_i)=(g_i,g_j)=1$ whenever $i\neq j$, there is an adjoining transformation $T_{A,B}$ such that
$T_{A,B}(\mathcal{L})$ is admissible and $g_i$ is the adjoining factor for each $L_i$ under $T_{A,B}$.
\end{enumerate}
\end{lemma}

\begin{proof}
Part (a) follows from definitions since $c_{j,i}L_j(m)-c_{i,j}L_i(m)=r_{i,j}$ implies $c_{j,i}g_j(L_j(Am+B)/g_j)-c_{i,j}g_i(L_i(Am+B)/g_i)=r_{i,j}$. %The second part follows from the first part since $f(g_ic_{i,j})$.
Part (b) follows from part (a) combined with \cref{lem:distance}. 
The authors
and their collaborators proved part (c) in \cite{GGPPSY}. In fact, it was shown that we can take $A=(g_1g_2\cdots g_k)^2$ and $B$ to be a solution of the congruences $L_i(B)\equiv g_i \pmod{g_i^2}$.
\end{proof}

% \begin{example}
% Consider again the $K_3$ relation diagram in \cref{dtriple} for the admissible triple $L_1=3m+2$, $L_2=10m+7$, and $L_3=4m+3$. All relation values in the diagram are $1$, but
% the diagram is not $\Omega$-compatible. However, we can use \cref{adjointhm}
% with $g_1=1$, $g_2=3$, $g_3=5$ to get an adjoining transformation $T_{225,113}$ which produces an admissible triple with the $\Omega$-compatible diagram seen in \cref{dtripleadjoin}. This allows one to conclude as in \cite{Gold2} that
% \begin{equation*}
%     \liminf_{n\rightarrow \infty} (S_{n+1}^{(4)} - S_n^{(4)}) = 1
% \end{equation*}
% \begin{figure}
%     \centering$$\xymatrixcolsep{2pc}\xymatrix{ & & 750m+379  \ar@{|->}[rrdd]^<<<<<<{(6)}^{\mbox{$1$}}^>>>{(25)} && \\
% & & & & \\
% 675m+341 \ar@{|->}[rruu]^<<<{(10)}^{\mbox{$1$}}^>>>>>>{(9)} \ar@{|->}[rrrr]_<<<{(4)}_{\mbox{$1$}}_>>>{(15)} & &  & & 180m+91}$$
%     \caption{An $\Omega$-Compatible Relation Diagram of an Admissible Triple.}
%     \label{dtripleadjoin}
% \end{figure}
% \end{example}

\section{Proofs of Theorems 2 to 4}

In this section we use \cref{pattern,adjointhm} to establish proofs for our main results.

\begin{proof}[Proof of \cref{E3}] First, we assume $\mathrm{EH}(\mathcal{P}, \mathcal{E}_2)$. Start with the admissible $5$-tuple $\mathcal{L}$ from \cref{diagram}: $2m + 1$, $3m + 2$, $6m + 5$, $6m + 7$, $3m + 4$. Recall that the diagram on $\mathcal{L}$ in \cref{fig:P3} is consistent with relation values ranging from $1$ to $5$ and relation coefficients $c_1=3$, $c_2 = 2$, $c_3=1$, $c_4=1$, $c_5=2$. We can apply a transformation to get
a new admissible tuple $T_{35,0}(\mathcal{L})$: $L_1'=70m+ 1$, $L_2'=105m+ 2$, $L_3'=42m+ 1$, $L_4'=30m+ 1$, $L_5'=105m+ 4$. The associated consistent diagram is both $\omega$- and $\Omega$-compatible since all relation coefficients are now prime: $c_1'=3$, $c_2'=2$, $c_3'=5$, $c_4'=7$, $c_5'=2$. See \cref{fig:E3}.
\begin{figure}
  \begin{tikzpicture}
    \node[
      regular polygon,
      regular polygon sides=5,
      minimum width=80mm,
    ] (L) {}
      (L.corner 1) node (L1) {$70m+ 1$}
      (L.corner 2) node (L2) {$105m + 2$}
      (L.corner 3) node (L3) {$42m+ 1$}
      (L.corner 4) node (L4) {$30m+1$}
      (L.corner 5) node (L5) {$105m+ 4$}
    ;
    \draw [thick, |->, shorten >= 6pt, shorten <= 6pt] (L1) -- (L2) node[midway,above left] {1} node[pos=0.1,above left] {\scriptsize (3)} node[pos=0.9,above left] {\scriptsize (2)};
    \draw [thick, |->, shorten >= 6pt, shorten <= 6pt] (L1) -- (L3) node[midway,above left] {2} node[pos=0.13,above left] {\scriptsize (3)} node[pos=0.9,above left] {\scriptsize (5)};
    \draw [thick, |->, shorten >= 6pt, shorten <= 6pt] (L1) -- (L4) node[midway,above right] {4} node[pos=0.13,above right] {\scriptsize (3)} node[pos=0.9,above right] {\scriptsize (7)};
    \draw [thick, |->, shorten >= 6pt, shorten <= 6pt] (L1) -- (L5) node[midway,above right] {5} node[pos=0.1,above right] {\scriptsize (3)} node[pos=0.9,above right] {\scriptsize (2)};
    \draw [thick, |->, shorten >= 6pt, shorten <= 6pt] (L2) -- (L3) node[midway,below left] {1} node[pos=0.05,below left] {\scriptsize (2)} node[pos=0.85,below left] {\scriptsize (5)};
    \draw [thick, |->, shorten >= 6pt, shorten <= 6pt] (L2) -- (L4) node[midway,below left] {3} node[pos=0.11,below left] {\scriptsize (2)} node[pos=0.9,below left] {\scriptsize (7)};
    \draw [thick, |->, shorten >= 6pt, shorten <= 6pt] (L2) -- (L5) node[midway,above] {4} node[pos=0.1,above] {\scriptsize (2)} node[pos=0.9,above] {\scriptsize (2)};
    \draw [thick, |->, shorten >= 4pt, shorten <= 4pt] (L3) -- (L4) node[midway,below] {2} node[pos=0.11,below] {\scriptsize (5)} node[pos=0.9,below] {\scriptsize (7)};
    \draw [thick, |->, shorten >= 6pt, shorten <= 6pt] (L3) -- (L5) node[midway,below right] {3} node[pos=0.11,below right] {\scriptsize (5)} node[pos=0.9,below right] {\scriptsize (2)};
    \draw [thick, |->, shorten >= 6pt, shorten <= 6pt] (L4) -- (L5) node[midway,below right] {1} node[pos=0.13,below right] {\scriptsize (7)} node[pos=0.95,below right] {\scriptsize (2)};
  \end{tikzpicture}
    \caption{An $\omega$- and $\Omega$-Compatible Relation Diagram}
    \label{fig:E3}
\end{figure}
Thus \cref{pattern} implies that there are integers $a,b$ with $1\leq a < b\leq 5$ such that $\omega(x)=\omega(x+a)=\omega(x+b)=\Omega(x)=\Omega(x+a) = \Omega(x+b)=3$ for infinitely many $x$. This proves
\[\liminf_{n\rightarrow \infty} (q^{(3)}_{n+2} - q^{(3)}_n)\leq 5. \]
%In fact, the diagram tells us that $\{a,b\}$ must be one of the sets $\{1,2\}$, $\{1,3\}$, $\{1, 4\}$, $\{1,5\}$, $\{2,3\}$, $\{2,4\}$, $\{2,5\}$, $\{3,4\}$, $\{4,5\}$ (so only missing $\{3,5\}$)

To get the unconditional result about small gaps between three $E_3$-numbers, we begin with a non-admissible 10-tuple $\mathcal{L}$ of monic forms: $m+4$, $m+5$, $m+7$, $m+8$, $m+9$, $m+11$, $m+13$, $m+16$, $m+17$, $m+19$. The trivial diagram on $\mathcal{L}$ with all relation coefficients equal to $1$ is consistent with relation values ranging from $1$ to $15$. We apply a transformation making use of a primorial $19\# = 2\cdot 3 \cdot 5 \cdot 7 \cdot 11 \cdot 13 \cdot 17 \cdot 19 = 9699690$ to get an admissible tuple $\mathcal{L}'=T_{19\#,0}(\mathcal{L})$:
\begin{align*}
    L_1'=4849845m+ 2, L_2'=1939938m+ 1, \\
    L_3'=1385670m+ 1, L_4'=4849845m+ 4, \\
    L_5'=3233230m+ 3, L_6'=881790m+ 1, \\
    L_7'=746130m+ 1, L_8'=4849845m+ 8, \\
    L_9'=570570m+ 1, L_{10}'=510510m+1.
\end{align*}
The associated consistent diagram is both $\omega$- and $\Omega$-compatible since all relation coefficients are prime: $c_1'=2$, $c_2'=5$, $c_3'=7$, $c_4'=2$, $c_5'=3$, $c_6'=11$, $c_7'=13$, $c_8'=2$, $c_9'=17$, $c_{10}'=19$. Thus \cref{pattern} implies that there are integers $a,b$ with $1\leq a < b\leq 15$ such that $\omega(x)=\omega(x+a)=\omega(x+b)=\Omega(x)=\Omega(x+a) = \Omega(x+b)=3$ for infinitely many $x$. This proves
\[\liminf_{n\rightarrow \infty} (q^{(3)}_{n+2} - q^{(3)}_n)\leq 15. \]
% To make the above tuple, I wrote down a list of ten positive integers a_1, ..., a_{10} in increasing order with the following properties: (1) no a_i is divisible by more than two primes from the list 2,3,5,7, (2) if a_i has the smallest index i such that a_i is divisible by a prime p from the list 2,3,5,7, then any larger index j > i where a_j is also divisible by p has the property that p^2|a_j if and only if p^2|a_i. Now we admissiblize the monic tuple m + a_1, ..., m + a_{10} by substituting m with 2*3*5*7m and then reducing forms. Properties (1) and (2) together give you admissibility of the new tuple and also ensure that relation coefficients (which were all 1 for the monic forms) get either a 1 or a prime adjoined to them.
\end{proof}

\begin{proof}[Proof of \cref{S3,Div}]
%First, we assume $\mathrm{EH}(\mathcal{P}, \mathcal{E}_2)$.
We begin with the non-admissible monic $k$-tuple $\mathcal{L}$: $L_i = m+i$ for $i=1, \ldots, k$. The trivial diagram on $\mathcal{L}$ with all relation coefficients equal to $1$ is consistent with relation values ranging from $1$ to $k-1$.  We apply a transformation making use of the least common multiple $A(k)=[1,2,\ldots,k]$ to get an admissible tuple $\mathcal{L}'=T_{A(k),0}(\mathcal{L})$ which satisfies a consistent diagram with edges
\begin{center}
\begin{tikzpicture}
  \coordinate [label={[xshift=0cm, yshift=-.3cm]$L_i'$}] (L1) at (-1.5cm,0cm);
  \coordinate [label={[xshift=0cm, yshift=-.35cm]$L_j'$}] (L2) at (1.5cm,0cm);
    ;
    \draw [thick, |->, shorten >= 9pt, shorten <= 9pt] (L1) -- (L2) node[midway,above] {$j-i$} node[pos=0.17,above] {\scriptsize $(i)$} node[pos=0.83,above] {\scriptsize $(j)$};
  \end{tikzpicture}
\end{center}
whenever $i<j$. We will adjoin factors $g_i$ by \cref{adjointhm} in various ways to maintain admissibility and get relation coefficients $c_i'=ig_i$ in our consistent diagram.
Note that the linear coefficients $a_i'=A(k)/i$ and the determinants $a_i'b_j'-a_j'b_i' = A(k)(j-i)/(ij)$ are only divisible by primes less than or equal to $k$, so we take our $g_i$ to consist of primes greater than $k$.
First we assume $\mathrm{EH}(\mathcal{P}, \mathcal{E}_2)$, so we may take $k=5$ here by \cref{thm:nu=2}.
Now adjoin $g_1=7$ and $g_i=1$ otherwise to get an admissible $5$-tuple $T_{7^2,5}(\mathcal{L}')$: $420m+ 43$, $1470m+ 151$, $980m+ 101$, $735m+ 76$, $588m+ 61$.
The associated consistent diagram is $\omega$-compatible with relation coefficients satisfying $\omega(ig_i)=1$ for all $i$, so \cref{pattern}
implies there are integers $a,b$ with $1\leq a<b \leq 4$ such that $\omega(x)=\omega(x+a)=\omega(x+b)=3$ for infinitely many $x$. This proves
\[\liminf_{n\rightarrow \infty} (s^{(3)}_{n+2} - s^{(3)}_n)\leq 4. \]
Next, we adjoin $g_1=7\cdot 11$, $g_2=13$, $g_3=17$, $g_4=1$, and $g_5=19$ to get an admissible $5$-tuple $T_{A,B}(\mathcal{L}')$ with $A=323323^2$ and $B=97650202718$:
\begin{align*}
81457996620m+ 76091067053,\\
241240989990m+ 225346621657, \\
122985602740m+ 114882591433, \\
1568066434935m+ 1464753040771, \\
66023849892m+ 61673812243.
\end{align*}
The associated consistent diagram is $\Omega$-compatible with relation coefficients satisfying $\Omega(ig_i)=2$ for all $i$, so \cref{pattern}
implies there are integers $a,b$ with $1\leq a<b \leq 4$ such that $\Omega(x)=\Omega(x+a)=\Omega(x+b)=4$ for infinitely many $x$. This proves
\[\liminf_{n\rightarrow \infty} (S^{(4)}_{n+2} - S^{(4)}_n)\leq 4. \]
Recall that $h(n)=$ least positive integer with the same exponent pattern as $n$.
We can get an $h$-compatible diagram (i.e., a common exponent pattern for all relation coefficients) by taking $g_1=7^2\cdot 11$, $g_2=13^2$, $g_3=17^2$, $g_4=19$, $g_5=23^2$ to get an admissible 5-tuple $T_{A,B}(\mathcal{L}')$ with $A=264595580249^2$ and $B=46136207543205182050716$:
\begin{align*}
7793412365933766111540m+ 5135755941729704866499, \\
12427956406030473177870m +8189859327196186162849, \\
4845039521612802692180m +3192817131017659657489, \\
55271700858398683343685m+ 36423321744635670040039, \\
1588147170222419634828m + 1046568035006544772417.
\end{align*}
The associated consistent diagram has relation coefficients satisfying $h(ig_i)=2^2\cdot 3^1$ for all $i$, so \cref{pattern}
implies there are integers $a,b$ with $1\leq a<b \leq 4$ such that $h(x)=h(x+a)=h(x+b)=2^2\cdot 3^1 \cdot 5^1 \cdot 7^1$ for infinitely many $x$, so here $x$, $x+a$, $x+b$ all have exponent pattern $\{2,1,1,1\}$ with $d(x)=d(x+a)=d(x+b)=24$.

To get unconditional results, we may take $k=10$ by \cref{thm:nu=2}. Now adjoin $g_1=11\cdot 13$, $g_2=17$, $g_3=19$, $g_4=23$, $g_5=29$, $g_6=1$, $g_7=31$, $g_8=37$, $g_9=41$, $g_{10}=1$ to get an admissible $10$-tuple whose associated consistent diagram is $\omega$-compatible with relation coefficients satisfying $\omega(ig_i)=2$ for all $i$, so \cref{pattern}
implies there are integers $a,b$ with $1\leq a<b \leq 9$ such that $\omega(x)=\omega(x+a)=\omega(x+b)=4$ for infinitely many $x$. This proves
\[\liminf_{n\rightarrow \infty} (s^{(4)}_{n+2} - s^{(4)}_n)\leq 9. \]
Next, we adjoin $g_1=11^3$, $g_2=13^2$, $g_3=17^2$, $g_4=19$, $g_5=23^2$, $g_6=29$, $g_7 = 31^2$, $g_8=1$, $g_9= 37$, $g_{10}=41$ to get an admissible $10$-tuple whose associated consistent diagram is $\Omega$-compatible with relation coefficients satisfying $\Omega(ig_i)=3$ for all $i$, so \cref{pattern}
implies there are integers $a,b$ with $1\leq a<b \leq 9$ such that $\Omega(x)=\Omega(x+a)=\Omega(x+b)=5$ for infinitely many $x$. This proves
\[\liminf_{n\rightarrow \infty} (S^{(5)}_{n+2} - S^{(5)}_n)\leq 9. \]
Lastly, we adjoin $g_1=11^3\cdot 13^2 \cdot 17\cdot 19$, $g_2=23^3\cdot 29^2\cdot 31$, $g_3=37^3\cdot 41^2\cdot 43$, $g_4=47^3\cdot 53\cdot 59$, $g_5=61^3\cdot 67^2 \cdot 71$, $g_6=73^3\cdot 79^2$, $g_7=83^3\cdot 89^2 \cdot 97$, $g_8=101^2\cdot 103 \cdot 107$, $g_9=109^3\cdot 113\cdot 127$, $g_{10}=131^3\cdot 137^2$ to get an admissible $10$-tuple whose consistent diagram is $h$-compatible with relation coefficients satisfying $h(ig_i)=2^3\cdot 3^2 \cdot 5^1\cdot 7^1$ for all $i$, so \cref{pattern}
implies there are integers $a,b$ with $1\leq a<b \leq 9$ such that $h(x)=h(x+a)=h(x+b)=2^3\cdot 3^2 \cdot 5^1 \cdot 7^1\cdot 11^1 \cdot 13^1$ for infinitely many $x$, so here $x$, $x+a$, $x+b$ all have exponent pattern $\{3,2,1,1,1,1\}$ with $d(x)=d(x+a)=d(x+b)=192$.
\end{proof}

\section{The GGPY Sieve and Proof of Theorem 6}\label{ggpy}

In this last section, we use the GGPY sieve of \cite{Gold1} to
prove \cref{thm:nu=2}. This will complete the proofs of \cref{E2,E3,S3,Div}. We seek to establish when there are three sifted $E_2$-numbers simultaneously infinitely often in an admissible $k$-tuple. 
The conjectural versus conditional results are distinguished by a common level of distribution for primes and $E_2$-numbers.

\begin{definition}
Define an error $E(x;a,b) = \pi(x;a,b) - \mathrm{li}(x)/\varphi(a)$ for the asymptotic in \cref{Dir}.
We say the primes have a \emph{level of distribution} $\vartheta$ if for all $A>0$ there is a constant $C$ (depending on $A$) such that
\begin{align*}
    \sum_{a\leq x^{\vartheta}/(\log x)^C} \max_{(a,b)=1} \left|E(x;a,b)\right| \ll_A \frac{x}{(\log x)^A}
\end{align*}
We can define a level of distribution for $E_2$-numbers in a completely analogous way.
\end{definition}

\begin{theorem}[Bombieri-Vinogradov, \cite{Bomb}, \cite{Vino}]
The primes have a level of distribution $1/2$.
\end{theorem}

\begin{theorem}[Motohashi, \cite{Moto}]
The $E_2$-numbers have a level of distribution $1/2$.
\end{theorem}

Friedlander and Granville \cite{Fried} showed that 
the primes do not have a level of distribution $1$. However,
the following conjecture, which first appeared for primes in \cite{Elli}, may still hold.

\begin{conjecture}[Elliot-Halberstam, $\mathrm{EH}(\mathcal{P}, \mathcal{E}_2)$] Primes and $E_2$-numbers have a common level of distribution $\vartheta$ for every $\vartheta < 1$.
\end{conjecture}

We define an indicator function for sifted $E_2$-numbers as follows. Take $\beta(n) = 1$ if $n = p_1p_2$ where each $p_i$ is a prime with $N^{\eta}< p_1 \leq N^{1/2} < p_2$ where $\eta \in (0, 1/4]$, and take $\beta(n)=0$ otherwise. Here $N$ is a real number which can be taken arbitrarily large.

Let $\mathcal{L}$ denote an admissible $k$-tuple
of linear forms $L_i(m) = a_im+b_i$, and define a counting
function
\begin{align*}
    \mathcal{S} = \sum_{N<n\leq 2N} \left(\sum_{j=1}^k \beta(L_j(n))-\nu\right)\left(\sum_{d|P_{\mathcal{L}}(n)} \lambda_d\right)^2
\end{align*}
where $P_{\mathcal{L}}(n) = L_1(n)L_2(n)\cdots L_k(n)$ and the $\lambda_d$ are chosen carefully so that the sum can be estimated with integrals of 
polynomials. Note that $n$ contributes a positive 
amount to $\mathcal{S}$ when at least $\nu + 1$ of $L_1(n)$, $L_2(n)$, $\ldots$, $L_k(n)$ are sifted 
$E_2$-numbers. Thus if $\mathcal{S}$ is positive
for sufficiently large $N$, we must have $\nu+1$ forms among our $k$-tuple which simultaneously assume sifted $E_2$-numbers as values infinitely often. There is a normalization $\mathcal{L}'=T_{A,B}(\mathcal{L})$ where $\mathcal{L}'$ is still admissible and all the adjoining factors are $1$ so that if the associated sum $\mathcal{S}'$ is positive for large $N$, then $\mathcal{S}$ will be positive for large $N$ as well. The normalization is made since $\mathcal{S}'$ is easier to estimate than $\mathcal{S}$. In particular, the authors in \cite{Gold1} showed that whenever a quantity $J$ is positive, we have
\begin{align*}
    \mathcal{S}'\sim \frac{\mathfrak{S}(\mathcal{L}')N (\log R)^k J}{(k-1)!}
\end{align*}
as $N\rightarrow \infty$ where $R\leq N^{1/2}$ tends to infinity. The quantity $J$ is given as a sum of integrals of polynomials $P(x)$ and $\tilde{P}(x)=\int_0^x P(t)\, dt$:
\begin{align*}
J&=\frac{k(k-1)}{B} \cdot (J_1+J_2+J_3) - \nu \cdot J_0, \\
J_0 &= \int_0^1 P(1-x)^2x^{k-1}\, dx, \\
J_1 &= \int_{B\eta}^1 \frac{B}{y(B-y)}\int_0^{1-y} \left(\tilde{P}(1-x)-\tilde{P}(1-x-y)\right)^2 x^{k-2} \, dx \,dy, \\
J_2 &= \int_{B\eta}^1 \frac{B}{y(B-y)}\int_{1-y}^1 \tilde{P}(1-x)^2 x^{k-2}\,dx \,dy, \\
J_3 &= \int_1^{B/2} \frac{B}{y(B-y)} \int_0^1 \tilde{P}(1-x)^2x^{k-2} \,dx \,dy.
\end{align*}
Here $B = 2/\vartheta$ where $\vartheta$ is a common level of distribution for primes and $E_2$-numbers. Thus to get small gaps between $\nu+1=3$ sifted $E_2$-numbers, we need to show that $J>0$ for some choices of $B$ (either $B=4$ unconditionally or $B=2+\epsilon$ assuming $\mathrm{EH}(\mathcal{P}, \mathcal{E}_2)$), $\eta\in (0,1/4]$, and some polynomial $P(x)$.
\begin{proof}[Proof of \cref{thm:nu=2}]
For $\nu=2$ and $B=4$, we can take $k=10$, $\eta = 1/340$, and $P(x) = \frac{3}{20}  + \frac{3}{5}x + 10x^2$. Then
{\footnotesize
 \begin{align*}
    J_0 &= \frac{18549}{800800} = 0.02316308\ldots,\\
    J_1 &= \frac{2113287710672781837420508478315754592524}{105076848611852709873077392578125} -\frac{1262566905669}{17875}\log\left(\frac{113}{85} \right)\\
    &= 0.00063269\ldots, \\
    J_2 &= -\frac{79584691575671328932238080587887183154}{3967936940587444988214111328125} + \frac{12980396496724}{184275} \log\left(\frac{113}{85} \right) \\
    &= 0.00074896\ldots,\\
    J_3 &= \frac{911}{1474200} \log(3) = 0.00067890\ldots, \\
    J &= \frac{8719967520249406350967107646046792519503}{
7061164226716502103470800781250000} -\frac{1953628194503}{450450}\log\left(\frac{113}{85}\right) \\*
&+\frac{911}{65520} \log(3) = 0.00003645\ldots > 0.
 \end{align*}
 }
Now assume $\mathrm{EH}(\mathcal{P}, \mathcal{E}_2)$. For $\nu=2$ and $B=201/100>2$, we can take $k=5$, $\eta = 1/340$, and $P(x) = \frac{3}{4}  + 6x + 10x^2$. Then
{\footnotesize
\begin{align*}
    J_0 &= \frac{7487}{5040} = 1.48551587\ldots,\\
    J_1 &= -\frac{1185867362212062499391309326339523040406768636033}{6426286514402549760000000000000000000000000000} \\*
    &+ \frac{75466092079924833449781}{280000000000000000000}\log\left(\frac{68139}{34340}\right) = 0.15294205\ldots, \\
    J_2 &= \frac{42290186710920567072992744924095611611325091277}{229510232657233920000000000000000000000000000} \\*
    &- \frac{18808406922710397486573}{70000000000000000000}\log\left(\frac{68139}{34340}\right) = 0.14486877\ldots,\\
    J_3 &= \frac{1747}{12096}\log\left(101/100\right) = 0.00143710\ldots, \\
    J &= -\frac{54641056436520615648365200967849481935314631}{9639429771603824640000000000000000000000000} + 
 \frac{218375}{151956} \log\left(\frac{101}{100}\right) \\*
 &+ 
 \frac{1156539249170365689}{140000000000000000}\log\left(\frac{68139}{34340}\right) = 0.00655959\ldots > 0.
 \end{align*}
 }
 \end{proof}

\newcommand{\etalchar}[1]{$^{#1}$}
\providecommand{\bysame}{\leavevmode\hbox to3em{\hrulefill}\thinspace}
\providecommand{\MR}{\relax\ifhmode\unskip\space\fi MR }
% \MRhref is called by the amsart/book/proc definition of \MR.
\providecommand{\MRhref}[2]{%
  \href{http://www.ams.org/mathscinet-getitem?mr=#1}{#2}
}
\providecommand{\href}[2]{#2}

\end{document}